\documentclass{article}

\usepackage{
	amsmath, 
	amsthm, 
	amssymb,
	fancyhdr, 
	setspace,
	tikz-cd,
}
\usepackage{upgreek}

\newcommand{\za}{\alpha}
\newcommand{\zb}{\beta}

\newcommand{\zd}{\delta}
\newcommand{\zg}{\gamma}

\usepackage{lscape}
\usepackage{color}
\usepackage[color,matrix,arrow]{xy}
\usepackage{setspace}
\xyoption{all}

\DeclareMathOperator{\Ext}{Ext}
\DeclareMathOperator{\Hom}{Hom}
\DeclareMathOperator{\pd}{pd}
\DeclareMathOperator{\id}{id}
\DeclareMathOperator{\add}{add}
\DeclareMathOperator{\Gen}{Gen}
\DeclareMathOperator{\End}{End}
\DeclareMathOperator{\ind}{ind}

\usepackage{avant}
\usepackage{amsmath, amssymb, amsfonts, color, tikz, bbm, txfonts, euscript}
\usepackage{amssymb}
\usepackage{enumerate}
\usepackage{bm}
\newtheorem{theorem}{Theorem}[section]
\newtheorem{lemma}[theorem]{Lemma}

\newtheorem{prop}[theorem]{Proposition}
\newtheorem{cor}[theorem]{Corollary}
\theoremstyle{definition}
\newtheorem{mydef}[theorem]{Definition}
\newtheorem{example}[theorem]{Example}

\begin{document}

\thispagestyle{empty}

\title{$\uptau$-RIGID MODULES FROM TILTED TO CLUSTER-TILTED ALGEBRAS}
\author{Stephen Zito}
        
\maketitle

\begin{abstract}
We study the module categories of a tilted algebra $C$ and the corresponding cluster-tilted algebra $B=C\ltimes E$ where $E$ is the $C$-$C$-bimodule $\Ext_C^2(DC,C)$.  In particular, we study which $\tau_C$-rigid $C$-modules are also $\tau_B$-rigid $B$-modules. 
\end{abstract}

\section{Introduction} 
We are interested in studying the representation theory of cluster-tilted algebras which are finite dimensional associative algebras that were introduced by Buan, Marsh, and Reiten in $\cite{BMR}$ and, independently, by Caldero, Chapoton, and Schiffler in $\cite{CCS}$ for type $\mathbb{A}$.
\par
One motivation for introducing these algebras came from Fomin and Zelevinsky's cluster algebras $\cite{FZ}$.  Cluster-tilted algebras are the endomorphism algebras of the so-called tilting objects in the cluster category of $\cite{BMMRRT}$.  Many people have studied cluster-tilted algebras in this context, see for example $\cite{BMR,BMR2,BMR3,CCS2,KR}$.
\par
The second motivation came from classical tilting theory, see $\cite{H}$.  Tilted algebras are the endomorphism algebras of tilting modules over hereditary algebras, whereas cluster-tilted algebras are the endomorphism algebras of cluster-tilting objects over cluster categories of hereditary algebras.  This similarity in the two definitions lead to the following precise relation between tilted and cluster-tilted algebras, which was established in $\cite{ABS}$.
\par
There is a surjective map
\[
\{\text{tilted}~\text{algebras}\}\longmapsto \{\text{cluster-tilted}~\text{algebras}\}
\]
\[
C\longmapsto B=C\ltimes E
\]
where $E$ denotes the $C$-$C$-bimodule $E=\text{Ext}_C^2(DC,C)$ and $C\ltimes E$ is the trivial extension.
\par
This result allows one to define cluster-tilted algebras without using the cluster category.  It is natural to ask how the module categories of $C$ and $B$ are related.  In this work, we investigate the $\tau$-rigidity of a $C$-module when the same module is viewed as a $B$-module via the standard embedding.  We let $M$ be a right $C$-module and define a right $B=C\ltimes E$ action on $M$ by 
\[
M\times B\rightarrow M~~,~~(m,(c,e))\mapsto mc.
\]
\par
 Our main results deal with $C$-modules that satisfy $\text{Hom}_C(M,\tau_CM)=0$ otherwise know as $\tau_C$-$\it{rigid~modules}$.  We show the following.
 \begin{theorem}
 \label{Main1}
 Let M be a partial tilting C-module.  Then M is $\tau_B$-rigid if and only if $\emph{Hom}_C(\tau_C^{-1}\Omega_C^{-1}M,\Gen M)=0$.
 \end{theorem}
As a consequence, we prove the following.
  \begin{cor} 
  Let M be an indecomposable $\tau_C$-rigid module.  Then M is $\tau_B$-rigid if and only if $\emph{Hom}_C(\tau_C^{-1}\Omega_C^{-1}M,\Gen M)=0.$
 \end{cor}
  We also prove necessary and sufficient conditions for a $\tau_C$-tilting module to be $\tau_B$-tilting.
  \begin{prop}  
  \label{Main Prop}
 Let  M be a $\tau_C$-tilting module.  Then M is $\tau_B$-tilting if and only if $\id_CM\leq1$.
\end{prop}

\section{Notation and Preliminaries}
We now set the notation for the remainder of this paper. All algebras are assumed to be finite dimensional over an algebraically closed field $k$.  Suppose $Q=(Q_0,Q_1)$ is a connected quiver without oriented cycles where $Q_0$ denotes a finite set of vertices and $Q_1$ denotes a finite set of oriented arrows.  By $kQ$ we denote the path algebra of $Q$.  If $C$ is a $k$-algebra then denote by $\mathop{\text{mod}}C$ the category of finitely generated right $C$-modules and by $\mathop{\text{ind}}C$ a set of representatives of each isomorphism class of indecomposable right $C$-modules.  Given $M\in\mathop{\text{mod}}C$, the projective dimension of $M$ in $\mathop{\text{mod}}C$ is denoted $\pd_CM$ and its injective dimension by $\id_CM$.  We denote by $\add M$ the smallest additive full subcategory of $\mathop{\text{mod}}C$ containing $M$, that is, the full subcategory of $\mathop{\text{mod}}C$ whose objects are the direct sums of direct summands of the module $M$.  We let $\tau_C$ and $\tau^{-1}_C$ be the Auslander-Reiten translations in $\mathop{\text{mod}}C$.  We let $D$ be the standard duality functor $\Hom_k(-,k)$.  Also, $\Omega_C M$ and $\Omega_C^{-1}M$ will denote the first syzygy and first cosyzygy of $M$.  Finally,  let gl.dim stand for the global dimension of an algebra.  
 
 \subsection{Tilted Algebras}
  Tilting theory is one of the main themes in the study of the representation theory of algebras.  Given a $k$-algebra $C$, one can construct a new algebra $B$ in such a way that the corresponding module categories are closely related.  The main idea is that of a tilting module.
   \begin{mydef} Let $C$ be an algebra.  A $C$-module $T$ is a $\emph{partial tilting module}$ if the following two conditions are satisfied: 
   \begin{enumerate}
   \item[($\text{1}$)] $\pd_CT\leq1$.
   \item[($\text{2}$)] $\Ext_C^1(T,T)=0$.
   \end{enumerate}
   A partial tilting module $T$ is called a $\emph{tilting module}$ if it also satisfies the following additional condition:
   \begin{enumerate}
   \item[($\text{3}$)] There exists a short exact sequence $0\rightarrow C\rightarrow T'\rightarrow T''\rightarrow 0$ in $\mathop{\text{mod}}C$ with $T'$ and $T''$ $\in \add T$.
   \end{enumerate}
   \end{mydef}
We recall that a $C$-module $M$ is $\it{faithful}$ if its right annihilator
   \[
   \mathop{\text{Ann}}M=\{c\in C~|~Mc=0\}.
   \]
   vanishes.  It follows easily from (3) that any tilting module is faithful.  We will need the following characterization of faithful modules.  Define $\mathop{Gen} M$ to be the class of all modules $X$ in $\mathop{\text{mod}}C$ generated by $M$, that is, the modules $X$ such that there exists an integer $d\geq0$ and an epimorphism $M^d\rightarrow X$ of $C$-modules.  Here, $M^d$ is the direct sum of $d$ copies of $M$.  Dually, we define $\mathop{Cogen}C$ to be the class of all modules $Y$ in $\mathop{\text{mod}}C$ cogenerated  by $M$, that is, the modules $Y$ such that there exist an integer $d\geq0$ and a monomorphism $Y\rightarrow M^d$ of $C$-modules.
   \begin{lemma}$\emph{\cite[VI,~Lemma~2.2.]{ASS}}$. 
   \label{faithful}
   Let C be an algebra and M a C-module.  The following are equivalent:
   \begin{enumerate}
   \item[$\emph{(a)}$] M is faithful.
   \item[$\emph{(c)}$] C is cogenerated by M.
   \item[$\emph{(d)}$] DC is generated by M.
   \end{enumerate}
   \end{lemma}
Tilting modules induce torsion pairs in a natural way.  We consider the restriction to a subcategory $\mathcal{C}$ of a functor $F$ defined originally on a module category, and we denote it by $F|_{\mathcal{C}}$.  Also, let $S$ be a subcategory of a category $\mathcal{C}$.  We say $S$ is a $full$ $subcategory$ of $\mathcal{C}$ if, for each pair of objects $X$ and $Y$ of $S$, $\Hom_S(X,Y)=\Hom_{\mathcal{C}}(X,Y)$. 
   \begin{mydef} A pair of full subcategories $(\mathcal{T},\mathcal{F})$ of $\mathop{\text{mod}}C$ is called a $\emph{torsion pair}$ if the following conditions are satisfied:
   \begin{enumerate}
   \item[(a)] $\text{Hom}_C(M,N)=0$ for all $M\in\mathcal{T}$, $N\in\mathcal{F}.$
   \item[(b)] $\text{Hom}_C(M,-)|_\mathcal{F}=0$ implies $M\in\mathcal{T}.$
   \item[(c)] $\text{Hom}_C(-,N)|_\mathcal{T}=0$ implies $N\in\mathcal{F}.$
   \end{enumerate}
   \end{mydef}
   Consider the following full subcategories of $\mathop{\text{mod}}C$ where $T$ is a tilting $C$-module.
 \[
 \mathcal{T}(T)=\{M\in\mathop{\text{mod}}C~|~ \text{Ext}_C^{1}(T,M)=0\}
 \]
 \[
 \mathcal{F}(T)=\{M\in\mathop{\text{mod}}C~|~\text{Hom}_C(T,M)=0\}
 \]
 Then $(\mathcal{T}(T),\mathcal{F}(T))$ is a torsion pair in $\mathop{\text{mod}}C$ called the $\it{induced~torsion~pair}$ of $T$.  Considering the endomorphism algebra $B=\End_CT$, there is an induced torsion pair, $(\mathcal{X}(T),\mathcal{Y}(T))$, in $\mathop{\text{mod}}B$.
 \[
 \mathcal{X}(T)=\{M\in\mathop{\text{mod}}B~|~M\otimes_BT=0\}
 \]
 \[
 \mathcal{Y}(T)=\{M\in\mathop{\text{mod}}B~|~\text{Tor}_1^B(M,T)=0\}
 \]
 
 We now state the definition of a tilted algebra.
 \begin{mydef} Let $C$ be a hereditary algebra with $T$ a tilting $C$-module.  Then the algebra $B=\End_CT$ is called a $\emph{tilted algebra}$.
 \end{mydef}
 
 The following proposition describes several facts about tilted algebras.  Let $C$ be an algebra and $M$, $N$ be two indecomposable $C$-modules.  A $\it{path}$ in $\mathop{\text{mod}}C$ from $M$ to $N$ is a sequence 
 \[
 M=M_0\xrightarrow{f_1}M_1\xrightarrow{f_2}M_2\rightarrow\dots\xrightarrow{f_s}M_s=N
 \]
 where $s\geq0$, all the $M_i$ are indecomposable, and all the $f_i$ are non-zero non-isomorphisms.  In this case, $M$ is called a $\it{predecessor}$ of $N$ in $\mathop{\text{mod}}C$ and $N$ is called a $\it{successor}$ of $M$ in $\mathop{\text{mod}}C$.  Also, we say a torsion pair $(\mathcal{T},\mathcal{F})$
 is $\it{split}$ if every indecomposable $C$-module belongs to either $\mathcal{T}$ or $\mathcal{F}$.
 
\begin{prop}$\emph{\cite[VIII,~Lemma~3.2.]{ASS}}$.
 \label{Tilting Prop}
 Let $C$ be a hereditary algebra, $T$ a tilting $C$-module, and $B=\End_CT$ the corresponding tilted algebra.  Then
 \begin{enumerate}
 \item[$\emph{(a)}$] $\mathop{\emph{gl.dim}}B\leq2$.
 \item[$\emph{(b)}$] For all $\text{M}\in\ind B$, $\id_BM\leq1$ or $\pd_BM\leq1$.
 \item[$\emph{(c)}$] For all $\text{M}\in\mathcal{X}(T)$, $\id_BM\leq1$.
 \item[$\emph{(d)}$] For all $\text{M}\in\mathcal{Y}(T)$, $\pd_BM\leq1$.
 \item[$\emph{(e)}$] $(\mathcal{X}(T),\mathcal{Y}(T))$ is split.
 \item[$\emph{(f)}$] $\mathcal{Y}(T)$ is closed under predecessors and $\mathcal{X}(T)$ is closed under successors.
 \end{enumerate}
 \end{prop}
 We also need the following characterization of split torsion pairs.
 \begin{prop}$\emph{\cite[VI,~Proposition~1.7]{ASS}}$
 \label{split}
 Let $(\mathcal{T},\mathcal{F})$ be a torsion pair in $\mathop{\emph{mod}}C$.  The following are equivalent:
 \begin{enumerate}
 \item[$\emph{(a)}$] $(\mathcal{T},\mathcal{F})$ is split.
 \item[$\emph{(b)}$] If $M\in\mathcal{T}$, then $\tau_C^{-1}M\in\mathcal{T}$.
 \item[$\emph{(c)}$] If $N\in\mathcal{F}$, then $\tau_CN\in\mathcal{F}$.
 \end{enumerate}
 \end{prop} 
 
 \subsection{Cluster categories and cluster-tilted algebras} Let $C=kQ$ and let $\mathcal{D}^b(\mathop{\text{mod}}C)$ denote the derived category of bounded complexes of $C$-modules.  The $cluster$ $category$ $\mathcal{C}_C$ is defined as the orbit category of the derived category with respect to the functor $\tau_\mathcal{D}^{-1}[1]$, where $\tau_\mathcal{D}$ is the Auslander-Reiten translation in the derived category and $[1]$ is the shift.  Cluster categories were introduced in $\cite{BMMRRT}$, and in $\cite{CCS}$ for type $\mathbb{A}$, and were further studied in $\cite{A,K,KR,P}$.  They are triangulated categories $\cite{K}$, that are 2-Calabi Yau and have Serre duality $\cite{BMMRRT}$.
 
 \par
 An object $T$ in $\mathcal{C}_C$ is called $cluster$-$tilting$ if $\text{Ext}_{\mathcal{C}_C}^1(T,T)=0$ and $T$ has $|Q_0|$ non-isomorphic indecomposable direct summands.  The endomorphism algebra $\End_{\mathcal{C}_C}T$ of a cluster-tilting object is called a $cluster$-$tilted$ $algebra$ $\cite{BMR}$.

\subsection{Relation Extensions}
 Let $C$ be an algebra of global dimension at most 2 and let $E$ be the $C$-$C$-bimodule $E=\text{Ext}_C^2(DC,C)$.  
 \begin{mydef} The $\emph{relation extension}$ of $C$ is the trivial extension $B=C\ltimes E$, whose underlying $C$-module structure is $C\oplus E$, and multiplication is given by $(c,e)(c^{\prime},e^{\prime})=(cc^{\prime},ce^{\prime}+ec^{\prime})$.
 \end{mydef}  
 
 Relation extensions were introduced in $\cite{ABS}$.  In the special case where $C$ is a tilted algebra, we have the following result.
 \begin{theorem}$\emph{\cite{ABS}}$.  Let C be a tilted algebra.  Then $B=C\ltimes\emph{Ext}_C^2(DC,C)$ is a cluster-tilted algebra.  Moreover all cluster-tilted algebras are of this form.
 \end{theorem}
 
  \subsection{Induction and coinduction functors}  A fruitful way to study cluster-tilted algebras is via induction and coinduction functors.  Recall, $D$ denotes the standard duality functor.
 
 \begin{mydef}
 Let $C$ be a subalgebra of $B$, then
 \[
 -\otimes_CB:\mathop{\text{mod}}C\rightarrow\mathop{\text{mod}}B
 \]
 is called the $\emph{induction functor}$, and dually
 \[
 D(B\otimes_CD-):\mathop{\text{mod}}C\rightarrow\mathop{\text{mod}}B
 \]
 is called the $\emph{coinduction functor}$.  Moreover, given $M\in\mathop{\text{mod}}C$, the corresponding induced module is defined to be $M\otimes_CB$, and the coinduced module is defined to be $D(B\otimes_CDM)$.  
 \end{mydef}
 We can say more in the situation when $B$ is a split extension of $C$.
 
 \begin{mydef}
 
 Let $B$ and $C$ be two algebras.  We say $B$ is a $\emph{split extension}$ of $C$ by a nilpotent bimodule $E$ if there exists a short exact sequence of $B$-modules
 \[
 0\rightarrow E\rightarrow B\mathop{\rightleftarrows}^{\mathrm{\pi}}_{\mathrm{\sigma}} C\rightarrow 0
\]
where $\pi$ and $\sigma$ are algebra morphisms, such that $\pi\circ\sigma=1_C$, and $E=\ker\pi$ is nilpotent.  
\end{mydef}
In particular, relation extensions are split extensions.  The next proposition shows a precise relationship between a given $C$-module and its image under the induction and coinduction functors.
\begin{prop}$\emph{\cite[Proposition~3.6]{SS}}$.  
\label{SS, Proposition 3.6}
Suppose B is a split extension of C by a nilpotent bimodule E.  Then, for every $M\in\mathop{\emph{mod}}C$, there exists two short exact sequences of B-modules:
\begin{enumerate}
\item[$\emph{(a)}$] $0\rightarrow M\otimes_CE\rightarrow M\otimes_CB\rightarrow M\rightarrow 0$
\item[$\emph{(b)}$] $0\rightarrow M\rightarrow D(B\otimes_CDM)\rightarrow D(E\otimes_CDM)\rightarrow 0$
\end{enumerate}

\end{prop}
 It was shown in $\cite{SS}(3.6)$ that, as a $C$-module, $M\otimes_CB\cong M\oplus (M\otimes_CE)$.

\subsection{Induced and coinduced modules in cluster-tilted algebras} In this section we cite several properties of the induction and coinduction functors particularly when $C$ is an algebra of global dimension at most 2 and $B=C\ltimes E$ is the trivial extension of $C$ by the $C$-$C$-bimodule $E=\text{Ext}_C^2(DC,C)$.  In the specific case when $C$ is also a tilted algebra, $B$ is the corresponding cluster-tilted algebra.
 \begin{prop}$\emph{\cite[Proposition~4.1]{SS}}$. 
 \label{SS Prop 4.1}
 Let C be an algebra of global dimension at most 2.  Then
 \begin{enumerate}
 \item[$\emph{(a)}$] $E\cong \tau_C^{-1}\Omega_C^{-1}C$.
 \item[$\emph{(b)}$] $DE\cong \tau_C\Omega_C(DC)$.
 \item[$\emph{(c)}$] $M\otimes_C E\cong\tau_C^{-1}\Omega_C^{-1}M$.
 \item[$\emph{(d)}$] $D(E\otimes_C DM)\cong\tau_C\Omega_C M$.
 
 \end{enumerate}
 \end{prop}
 
 The next two results use homological dimensions to extract information about induced and coinduced modules.
 \begin{prop}$\emph{\cite[Proposition~4.2]{SS}}$.  
 \label{SS, Proposition 4.2}
 Let C be an algebra of global dimension at most 2, and let $B=C\ltimes E$.  Suppose $M\in\mathop{\emph{mod}}C$, then
 \begin{enumerate}
 \item[$\emph{(a)}$] $\id_CM\leq1$ if and only if $M\otimes_CB\cong M$.
 \item[$\emph{(b)}$] $\pd_CM\leq1$ if and only if $D(B\otimes_CDM)\cong M$.
 \end{enumerate}
 \end{prop}
 
 The following holds when $C$ is tilted.
 \begin{lemma}$\emph{\cite[Lemma~4.5]{SS}}$
 \label{SS 4.5}
 Let $C$ be a tilted algebra.  Then for all $M\in\mathop{\text{mod}}C$
 \begin{enumerate}
  \item[$\emph{(a)}$] $\id_CM\otimes_CE\leq1$ 
 \item[$\emph{(b)}$] $\pd_CD(E\otimes_CDM)\leq1$
 \end{enumerate} 
 \end{lemma}

 The following lemma is used extensively.
 \begin{lemma}$\emph{\cite[Lemma~3.1]{Zito}}$
 \label{Homological Result}
 Let C be an algebra of global dimension equal to 2 and let M be a C-module.  Then, 
 \begin{enumerate}
 \item[$\emph{(a)}$] $\pd_CM\leq1$ if and only if $\emph{Hom}_C(\tau_C^{-1}\Omega_C^{-1}C,M)=0.$
 \item[$\emph{(b)}$] $\id_CM\leq1$ if and only if $\emph{Hom}_C(M,\tau_C\Omega_C (DC))=0.$
 \end{enumerate}
 \end{lemma}
 
The following corollary will be used in section 3.2.
 \begin{cor}
 \label{corollary homological}
 Suppose $\pd_CM\leq1$.  Then for any $N\in\mathop{\emph{mod}}C$, $\emph{Hom}_C(\tau_C^{-1}\Omega_C^{-1}N,M)=0.$
 \end{cor}
 \begin{proof}
 Let $f\!:P\rightarrow N$ be a projective cover of $N$ in $\mathop{\text{mod}}C$.  Apply the functor $-\otimes_CE$ to obtain a surjective morphism $f\otimes_C1_E\!:P\otimes_CE\rightarrow N\otimes_CE$.  Apply $\text{Hom}_C(-,M)$ to obtain the exact sequence
\[
0\rightarrow\text{Hom}_C(N\otimes_CE,M)\xrightarrow{\overline{f\otimes_C1_E}}\text{Hom}_C(P\otimes_CE,M).
\] 
 Now, Proposition $\ref{SS Prop 4.1}$ (c) says $N\otimes_CE\cong\tau_C^{-1}\Omega_C^{-1}N$ and $P\otimes_CE\cong\tau_C^{-1}\Omega_C^{-1}P$.  Thus, we have that $\text{Hom}_C(P\otimes_CE,M)=0$ by Lemma $\ref{Homological Result}$ (a) and conclude $\text{Hom}_C(\tau_C^{-1}\Omega_C^{-1}N,M)=0$.
\end{proof} 
The following main result from $\cite{Zito}$ is needed.
\begin{theorem}$\emph{\cite[Theorem~4.1]{Zito}}$ 
 \label{Main Ext}
 Let M be a rigid C-module with a projective cover $P_0\rightarrow M$ and an injective envelope $M\rightarrow I_0$ in $\mathop{\emph{mod}}C$.  
 \begin{enumerate}
 \item[$\emph{(a)}$] If $\emph{Hom}_C(\tau_C^{-1}\Omega_C^{-1}P_0,M)=0$, then M is a rigid B-module.
 \item[$\emph{(b)}$] If $\emph{Hom}_C(M,\tau_C\Omega_C I_0)=0$, then M is a rigid B-module.
 \end{enumerate}
 \end{theorem}

\subsection{$\tau$-rigid modules} 
 Following $\cite{AIR}$ we state the following definition.
 
 \begin{mydef}  A $C$-module $M$ is $\tau_C$-$\emph{rigid}$ if $\text{Hom}_C(M,\tau_CM)=0$.  A $\tau_C$-rigid module $M$ is $\tau_C$-$\emph{tilting}$ if the number of pairwise, non-isomorphic, indecomposable summands of $M$ equals the number of isomorphism classes of simple $C$-modules.
 \end{mydef}  
 
 It follows from the Auslander-Reiten formulas that any $\tau_C$-rigid module is rigid and the converse holds if the projective dimension is at most 1.  In particular, any partial tilting module is a $\tau_C$-rigid module, and any tilting module is a $\tau_C$-tilting module.  Thus, we can regard $\tau_C$-tilting theory as a generalization of  classic tilting theory. 
 \par
 The following theorem provides a characterization of $\tau_C$-rigid modules.
 \begin{prop}$\emph{\cite[Proposition~5.8]{AS}}$.
 \label{GenM Result}
   For X and Y in $\mathop{\emph{mod}}C$, $\emph{Hom}_C(X,\tau_CY)=0$ if and only if $\emph{Ext}_C^1(Y,\mathop{\emph{Gen}}X)=0$.
 \end{prop}
The following observations are useful.
 \begin{prop}$\emph{\cite[Proposition~2.4]{AIR}}$.  
 \label{Useful Observations}
 Let X be in $\mathop{\emph{mod}}C$ with a minimal projective presentation $P_1\overset{d_1}{\rightarrow} P_0\overset{d_0}{\rightarrow} X\rightarrow 0$.
 \begin{enumerate}
 \item[$\emph{(a)}$] For Y in $\mathop{\emph{mod}}C$, we have an exact sequence 
 \[
 \hspace{-.9cm}0\rightarrow\emph{Hom}_C(Y,\tau_CX)\rightarrow D\emph{Hom}_C(P_1,Y)\overset{D(d_1,Y)}{\longrightarrow}D\emph{Hom}_C(P_0,Y)\overset{D(d_0,Y)}{\longrightarrow}D\emph{Hom}_C(X,Y)\rightarrow 0.
 \]
 \item[$\emph{(b)}$] $\emph{Hom}_C(Y,\tau_CX)=0$ if and only if the morphism $\emph{Hom}_C(P_0,Y)\overset{(d_1,Y)}{\longrightarrow}\emph{Hom}_C(P_1,Y)$ is surjective.
 \item[$\emph{(c)}$] X is $\tau_C$-rigid if and only if the morphism $\emph{Hom}_C(P_0,X)\overset{(d_1,X)}{\longrightarrow}\emph{Hom}_C(P_1,X)$ is surjective.
 
 \end{enumerate}
  \end{prop}
 The following lemma is very useful in applications.  We need several preliminary definitions.  Let $U$ be a $C$-module.  We define 
 \[
 ^{\perp}(\tau_CU)=\{X\in\mathop{\text{mod}}C~|~\text{Hom}_C(X,\tau_CU)=0\}.  
 \]
 Also, we say a module $X\in\mathop{\text{Gen}}U$ is $\it{Ext}$-$\it{projective}$ if $\text{Ext}_C^1(X,\Gen U)=0$.  We denote by $P(\mathop{\text{Gen}}U)$ the direct sum of one copy of each indecomposable $\text{Ext}$-projective module in $\mathop{\text{Gen}}U$ up to isomorphism.  We say a morphism $f\!:A\rightarrow B$ is a $\it{left}$ $\mathop{\text{Gen}}M$-$\it{approximation}$ if $B\in\mathop{\text{Gen}}M$ and, whenever $g\!:A\rightarrow X$ is a morphism with $X\in\mathop{\text{Gen}}M$, there is some $h\!:B\rightarrow X$ such that $h\circ f=g$.  Moreover, it is called $\it{minimal}$ if any map $j\!:A\rightarrow A$ satisfying  $f\circ j=f$ is an isomorphism.
 
 \begin{lemma}$\emph{\cite[Lemma~2.20]{AIR}}$.  
 \label{Butt Hurt}
 Let T be a $\tau_C$-rigid module.  If U is a $\tau_C$-rigid module satisfying $^{\perp}(\tau_CT)\subseteq~^{\perp}(\tau_CU)$, then there is an exact sequence
 \[
 U\xrightarrow{f}T'\rightarrow A\rightarrow 0
 \]
 satisfying the following conditions.
 \begin{itemize}
 \item f is a minimal left $\mathop{\emph{Gen}}T$-approximation of U.
 \item $T'$ is in $\add T$, A is in $\add P(\mathop{\emph{Gen}}T)$, and $\add T'\cap\add A=0$.
 \end{itemize}
 \end{lemma}
  We will also need the following special cases of Lemma $\ref{Butt Hurt}$.  The first deals with the case $A=0$.
\begin{lemma}$\emph{\cite[Lemma~2.21]{AIR}}$
 \label{extra}
 Assume $A=0$.  Then $f:U\rightarrow T'$ induces an isomorphism $U/\langle e\rangle U\cong T'$ for a maximal idempotent $e$ of $C$ satisfying $eT=0$.  In particular, if $T$ is sincere, then $U\cong T'$.
 \end{lemma}
The second deals with the case $T$ is $\tau_C$-tilting.
 \begin{lemma}$\emph{\cite[Proposition~2.23]{AIR}}$.  
 \label{Butt Hurt 2}
 Let T be a $\tau_C$-tilting module.  Assume that U is a $\tau_C$-rigid module such that $\mathop{\emph{Gen}}T\subseteq~^{\perp}(\tau_CU)$.  Then there exists an exact sequence 
 \[
 U\overset{f}{\rightarrow}T^0\rightarrow T^1\rightarrow 0 
 \]
 such that 
 \begin{itemize}
 \item f is a minimal left $\mathop{\emph{Gen}}T$-approximation of U. 
 \item T$^0$ and T$^1$ are in $\add T$ and satisfy $\add T^0\cap \add T^1=0$.
 \end{itemize}
 \end{lemma}
 
 The following definition was introduced in $\cite{AIR}$.
  \begin{mydef}  A $C$-module $M$ is $\emph{support}~\tau_C$-$\emph{tilting}$ if there exists an idempotent $e$ of $C$ such that $M$ is a $\tau$-tilting $(C/\langle e\rangle)$-module.
 \end{mydef} 
It was shown in $\cite{AIR}$ that $\tau$-tilting modules are sincere.
\begin{prop}$\emph{\cite[Proposition~2.2]{AIR}}$ 
 \label{Sincere}
 $\tau$-tilting modules are precisely sincere support $\tau$-tilting modules.
 \end{prop}

 We now return to the situation where the algebra $B$ is a split extension of the algebra $C$ by a nilpotent bimodule $E$.  The induction functor can be used to derive information about the Auslander-Reiten translation of a $C$-module $M$ inside the module category of $B$.  The next theorem tells us exactly when the Auslander-Reiten translation remains the same, i.e., $\tau_CM\cong\tau_BM$ as $B$-modules.
 \begin{theorem}$\emph{\cite[Theorem~2.1]{AZ}}$.
 \label{AZ Main}
   Let M be an indecomposable non-projective C-module.  The following are equivalent:
 \begin{enumerate}
 \item[$\emph{(a)}$] The almost split sequences ending with M in $\mathop{\emph{mod}}C$ and $\mathop{\emph{mod}}B$ coincide.
 \item[$\emph{(b)}$] $\tau_CM\cong\tau_BM$.
 \item[$\emph{(c)}$] $\emph{Hom}_C(E,\tau_CM)=0$ and $M\otimes_CE=0$.
 \end{enumerate}
 \end{theorem}
 Having information about the Auslander-Reiten translation of an induced module is very useful.
 \begin{lemma}$\emph{\cite[Lemma~2.1]{AM}}$.  
 \label{AM, 2.1}
 Let $M$ be a $C$-module.  Then 
 \[
 \tau_B(M\otimes_CB)\cong\emph{Hom}_C(_BB_C,\tau_CM)\cong\tau_CM\oplus\emph{Hom}_C(E,\tau_CM) 
 \]
 where the isomorphisms are isomorphisms of C-modules.
 \end{lemma}
 Next, we state a result which gives information about $\Hom_B(-,\tau_B(M\otimes_CB))$ and $\Hom_B(M\otimes_CB,-)$.
  
 \begin{lemma}$\emph{\cite[Lemma~1.5]{Zito2}}$
\label{Adjunct}
 Let $M$ be a $C$-module, $M\otimes_CB$ the induced module, and let $X$ be any $B$-module.  Then we have 
\[
\emph{Hom}_B(X,\tau_B(M\otimes_CB))\cong\emph{Hom}_B(X,\emph{Hom}_C(_BB_C,\tau_CM)\cong\emph{Hom}_C(X\otimes_BB_C,\tau_CM)
\]
and
\[
\emph{Hom}_B(M\otimes_CB,X)\cong\emph{Hom}_C(M,\emph{Hom}_B(_CB_B,X)).
\]
\end{lemma}
We note that $-\otimes_BB_C$ and $\text{Hom}_B(_CB_B,-)$ are two expressions for the forgetful functor $\mathop{\text{mod}}B\rightarrow\mathop{\text{mod}}C$.
 
 Deducing information about $\tau_BM$ is generally more difficult but we have an answer in the following special case.
 \begin{lemma}$\emph{\cite[Corollary~1.3]{AZ}}$.  
 \label{AZ corollary}
 Assume $M\otimes_CE=0$, then we have 
 \[
 \tau_BM\cong\tau_CM\oplus \emph{Hom}_C(E,\tau_CM)
 \]
where the isomorphism is an isomorphism of $C$-modules. 
 \end{lemma}
 We also have the following important fact which will be used extensively.
 \begin{lemma}$\emph{\cite[Corollary~1.2]{AZ}}$.
 \label{AR submodule}
 $\tau_B(M\otimes_CB)$ is a submodule of $\tau_BM$.
 \end{lemma}
 Finally, we note the following lemma. 
   \begin{lemma}$\emph{\cite[Lemma~2.1]{AZ2}}$
 \label{My Projective Cover}
 Let $M$ be a $C$-module with $f\!:P_0\rightarrow M$ a projective cover in $\mathop{\emph{mod}}C$.  Suppose $g\!:P_0\otimes_CB\rightarrow P_0$ is a projective cover of $P_0$ in $\mathop{\emph{mod}}B$.  Then $f\circ g\!:P_0\otimes_CB\rightarrow M$ is a projective cover of $M$ in $\mathop{\emph{mod}}B$.
 \end{lemma}  
  
  \section{Main Results}
 We assume $C$ is an algebra of global dimension 2 and $B=C\ltimes E$ where $E$=Ext$_C^2(DC,C)$.  Specific cases will be explicitly stated.  We wish to use various homological dimensions to derive information about the $\tau_B$-rigidity of a $C$-module.  We begin with determining when the Auslander-Reiten translation of a $C$-module remains unchanged in $\mathop{\text{mod}}C$ and $\mathop{\text{mod}}B$, i.e., when is $\tau_CM\cong\tau_BM$ as $B$-modules.
 \subsection{Homological Dimensions and $\tau_B$-rigidity}
 \begin{prop} Let M be a C-module.  Then $\tau_CM\cong\tau_BM$ if and only if $\pd_C\tau_CM\leq1$ and $\id_CM\leq1$ \end{prop}
 \begin{proof}
 By Theorem $\ref{AZ Main}$, we know $\tau_CM\cong\tau_BM$ if and only if $\text{Hom}_C(E,\tau_CM)=0$ and $M\otimes_CE=0$.  Using Lemma $\ref{Homological Result}$, we know that $\pd_C\tau_CM\leq1$ if and only if $\text{Hom}_C(\tau_C^{-1}\Omega_C^{-1}C,\tau_CM)=0$.  Since Proposition $\ref{SS Prop 4.1}$ gives $E\cong\tau_C^{-1}\Omega_C^{-1}C$, this is equivalent to $\text{Hom}_C(E,\tau_CM)=0$.  Using Proposition $\ref{SS, Proposition 4.2}$, we have $M\otimes_CE=0$ if and only if $\id_CM\leq1$.  Our result follows. 
\end{proof}
  
\begin{prop}
 \label{ID}
 Let M be a  $\tau_C$-rigid C-module.  If $\id_CM\leq1$, then M is $\tau_B$-rigid.
 \end{prop}
 \begin{proof}
 Since $\id_CM\!\leq\!1$, Proposition $\ref{SS, Proposition 4.2}$ implies $M\otimes_CE=0$.  By Lemma $\ref{AZ corollary}$, we have $\tau_BM\cong\tau_CM\oplus\text{Hom}_C(E,\tau_CM)$ as $C$-modules.  Now, we want to show that $\text{Hom}_B(M,\tau_BM)=0$.  Since any $B$-module homomorphism is also a $C$-module homomorphism, it suffices to show that $\text{Hom}_C(M,\tau_CM)$ and $\text{Hom}_C(M,\text{Hom}_C(E,\tau_CM))$ are equal to 0.  Now, $\text{Hom}_C(M,\text{Hom}_C(E,\tau_CM))\cong\text{Hom}_C(M\otimes_CE,\tau_CM)$ by the adjoint isomorphism.  Since $M\otimes_CE=0$, we conclude $\text{Hom}_C(M,\text{Hom}_C(E,\tau_CM))=0$.  Certainly, $M$ being $\tau_C$-rigid implies $\text{Hom}_C(M,\tau_CM)=0$.  Thus, we conclude $M$ is $\tau_B$-rigid.
 \end{proof}
 \begin{prop} 
 \label{Induced}
 Let M be a $\tau_C$-rigid C-module.  If $\pd_C\tau_CM\leq1$, then the induced module $M\otimes_CB$ is $\tau_B$-rigid.
 \end{prop}
 \begin{proof}
 Consider the following short exact sequence guaranteed by Proposition $\ref{SS, Proposition 3.6}$ and Proposition $\ref{SS Prop 4.1}$.
 \[
 0\rightarrow\tau_C^{-1}\Omega_C^{-1}M\rightarrow M\otimes_CB\rightarrow M\rightarrow 0.
 \]
 Apply $\text{Hom}_B(-,\tau_B(M\otimes_CB))$ to obtain the exact sequence
 \[
 \text{Hom}_B(M,\tau_B(M\otimes_CB))\rightarrow\text{Hom}_B(M\otimes_CB,\tau_B(M\otimes_CB))\rightarrow\text{Hom}_B(\tau_C^{-1}\Omega_C^{-1}M,\tau_B(M\otimes_CB)).
 \]
 We wish to show that $\text{Hom}_B(M\otimes_CB,\tau_B(M\otimes_CB))=0$.  Using Lemma $\ref{AM, 2.1}$, we know that $\tau_B(M\otimes_CB)\cong\tau_CM\oplus\text{Hom}_C(E,\tau_CM)$ as $C$-modules.  Since $\pd_C\tau_CM\leq1$, Lemma $\ref{Homological Result}$ implies $\text{Hom}_C(E,\tau_CM)=0$.  Thus, $\tau_B(M\otimes_CB)\cong\tau_CM$.  Since $M$ is a $\tau_C$-rigid module, we have that $\text{Hom}_B(M,\tau_B(M\otimes_CB))=0$.  
 \par
Next, consider $f\!:P_0\rightarrow M$, a projective cover of $M$ in $\mathop{\text{mod}}C$.  Apply the functor $-\otimes_CE$ to obtain a surjective morphism $f\otimes_C1_E\!:P_0\otimes_CE\rightarrow M\otimes_CE$.  This gives a short exact sequence
\[
 0\rightarrow\ker f\otimes_C1_E\rightarrow P_0\otimes_CE\xrightarrow{f\otimes_C1_E}M\otimes_CE\rightarrow 0.
 \] 
 Apply $\text{Hom}_C(-,\tau_CM)$ to obtain the exact sequence
 \[ 
  0\rightarrow\text{Hom}_C(M\otimes_CE,\tau_CM)\xrightarrow{\overline{f\otimes_C1_E}}\text{Hom}_C(P_0\otimes_CE,\tau_CM).
 \] 
 We know from Proposition $\ref{SS Prop 4.1}$ that $P_0\otimes_CE\cong\tau_C^{-1}\Omega_C^{-1}P_0$ and $M\otimes_CE\cong\tau_C^{-1}\Omega_C^{-1}M$.  Thus, any non-zero morphism from $\tau_C^{-1}\Omega_C^{-1}M$ to $\tau_CM$ would imply a non-zero morphism from $\tau_C^{-1}\Omega_C^{-1}P_0$ to $\tau_CM$ because $\overline{f\otimes_C1_E}$ is injective.  Since $\pd_C\tau_CM\leq1$, this is a contradiction by Lemma $\ref{Homological Result}$.  Thus, $\text{Hom}_B(\tau_C^{-1}\Omega_C^{-1}M,\tau_B(M\otimes_CB))=0$.  Since we have shown that $\text{Hom}_B(M,\tau_B(M\otimes_CB))$ and $\text{Hom}_B(\tau_C^{-1}\Omega_C^{-1},\tau_B(M\otimes_CB))$ are equal to 0, we conclude $\text{Hom}_B(M\otimes_CB,\tau_B(M\otimes_CB))=0$.
 \end{proof}

\subsection{Partial Tilting Modules and $\tau_B$-rigidity}
 In this section, we examine partial tilting $C$-modules and their $\tau_B$-rigidity.  We begin with a sufficient condition for $M$ to be $\tau_B$-rigid where $B$ is a split extension of $C$ by a nilpotent bimodule $E$ and $M$ is $\tau_C$-rigid but not necessarily partial tilting.  This result was shown in $\cite{Zito2}$ but we include the proof for the benefit of the reader.

\begin{prop}$\emph{\cite[Proposition~3.1]{Zito2}}$.
\label{result}
If $\emph{Hom}_C(M\otimes_CE,\Gen M)=0$, then $M$ is $\tau_B$-rigid.
\end{prop}
\begin{proof}
By Proposition $\ref{SS, Proposition 3.6}$, we have the following short exact sequence in $\mathop{\text{mod}}B$ 

\[
0\rightarrow M\otimes_CE\rightarrow M\otimes_CB\rightarrow M\rightarrow 0.  
\]
Applying $\text{Hom}_B(-,\Gen M)$, we obtain an exact sequence
\[
\text{Hom}_B(M\otimes_CE,\Gen M)\rightarrow \text{Ext}_B^1(M,\Gen M)\rightarrow \text{Ext}_B^1(M\otimes_CB,\Gen M).
\]
First, we wish to show $\text{Ext}_B^1(M\otimes_CB,\Gen M)=0$.  We know from Proposition $\ref{GenM Result}$ this is equivalent to $\text{Hom}_B(M,\tau_B(M\otimes_CB))=0$.  By Lemma $ \ref{Adjunct}$ and the assumption that $M$ is $\tau_C$-rigid, $\text{Hom}_B(M,\tau(M\otimes_CB))\cong\text{Hom}_C(M,\tau_CM)=0$.  Next, we want to show $\text{Hom}_B(M\otimes_CE,\Gen M)=0$.  By restriction of scalars, any non-zero morphism from $M\otimes_CE$ to $\Gen M$ in $\mathop{\text{mod}}B$ would give a non-zero morphism in $\mathop{\text{mod}}C$, contrary to our assumption.  Thus, $\text{Hom}_B(M\otimes_CE,\Gen M)=0$.  We conclude $\text{Ext}_B^1(M,\Gen M)=0$ and Proposition $\ref{GenM Result}$ implies $M$ is $\tau_B$-rigid.
\end{proof}
 For the next result, we assume $C$ is an algebra of global dimension 2 and $B=C\ltimes E$ where $E=\text{Ext}_C^2(DC,C)$. 
 \begin{theorem} 
 \label{Partial Tilt Result}
 Let M be a partial tilting C-module such that $\pd_C\tau_CM\leq1$.  Then M is $\tau_B$-rigid if and only if $\emph{Hom}_C(\tau_C^{-1}\Omega_C^{-1}M,\Gen M)=0$.
 \end{theorem}
 \begin{proof}
 Assume $\text{Hom}_C(\tau_C^{-1}\Omega_C^{-1}M,\Gen M)=0$.  We know from Proposition $\ref{SS Prop 4.1}$ (c) that $M\otimes_CE\cong\tau_C^{-1}\Omega_C^{-1}M$.  By Proposition $\ref{result}$, $M$ is $\tau_B$-rigid.
\par
Assume $M$ is $\tau_B$-rigid.  Since $\pd_C\tau_CM\leq1$, we know $M\otimes_CB$ is a $\tau_B$-rigid module by Proposition $\ref{Induced}$.  Since $M\otimes_CB$ is $\tau_B$-rigid and $\tau_B(M\otimes_CB)$ is a submodule of $\tau_BM$ by Lemma $\ref{AR submodule}$, we have $^{\perp}(\tau_BM)\subseteq~^{\perp}(\tau_B(M\otimes_CB))$.  Thus, Lemma $\ref{Butt Hurt}$ guarantees an exact sequence 
\[
M\otimes_CB\xrightarrow{f}M'\xrightarrow{g} N\rightarrow0
\]
where $M'\in\add M$ and $N\in\add P(\Gen M)$.  Next, consider the short exact sequence
\[
0\rightarrow\ker g\xrightarrow{i} M'\xrightarrow{g}N\rightarrow 0.
\]  
We know that $f\!:M\otimes_CB\rightarrow\ker g$ is a surjective morphism.  Considering $f$ as a morphism of $C$-modules, we have a surjective morphism $f\!:M\oplus\tau_C^{-1}\Omega_C^{-1}M\rightarrow\ker g$ where the following decomposition $M\otimes_CB\cong M\oplus\tau_C^{-1}\Omega_C^{-1}M$ is given by Proposition $\ref{SS Prop 4.1}$.  Now, consider the $\text{Hom}$ space $\text{Hom}_C(\tau_C^{-1}\Omega_C^{-1}M,\ker g)$.  If this $\text{Hom}$ space were not equal to 0, then the injectivity of $i$ would imply a non-zero morphism from $\tau_C^{-1}\Omega_C^{-1}M$ to $M'$.  But $M'$ is partial tilting and we would have a contradiction to Corollary $\ref{corollary homological}$.  But we can not have a surjective morphism from $M$ to $\ker g$ because this would imply $\ker g\in\Gen M$ and would contradict $N\in\add P(\Gen M)$.  Thus, $N=0$ and we have a short exact sequence
\[
0\rightarrow\ker f\rightarrow M\otimes_CB\xrightarrow{f}M'\rightarrow 0.
\]     
Apply $\text{Hom}_B(-,\Gen M)$ to obtain an exact sequence
\[
0\rightarrow\text{Hom}_B(M',\Gen M)\xrightarrow{\overline{f}}\text{Hom}_B(M\otimes_CB,\Gen M)\rightarrow\text{Hom}_B(\ker f,\Gen M).
\]
Now, Lemma $\ref{Butt Hurt}$ says that $f$ is a left $\Gen M$-approximation of $M\otimes_CB$.  This implies that $\overline{f}$ is surjective and the exactness of the sequence further implies $\overline{f}$ is an isomorphism.  
Using the following short exact sequence guaranteed by Proposition $\ref{SS, Proposition 3.6}$ and Proposition $\ref{SS Prop 4.1}$.
 \[
 0\rightarrow\tau_C^{-1}\Omega_C^{-1}M\xrightarrow{h} M\otimes_CB\rightarrow M\rightarrow 0
 \]
, we apply $\text{Hom}_B(-,\Gen M)$ to obtain an exact sequence
\[
0\rightarrow\text{Hom}_B(M,\Gen M)\rightarrow\text{Hom}_B(M\otimes_CB,\Gen M)\xrightarrow{\overline{h}}\text{Hom}_B(\tau_C^{-1}\Omega_C^{-1}M,\Gen M)\rightarrow 0
\]     
 where $\text{Ext}_B^1(M,\Gen M)=0$ by Proposition $\ref{GenM Result}$.  Since $\overline{h}$ is a surjective morphism, given $a\in\text{Hom}_B(\tau_C^{-1}\Omega_C^{-1}M,\Gen M)$, there exists $b\in\text{Hom}_B(M\otimes_CB,\Gen M)$ such that $a=b\circ h$.  
\[
\begin{tikzcd}
 \tau_C^{-1}\Omega_C^{-1}M\arrow{d}[swap]{a=b\circ h}\arrow{r}{h} & M\otimes_CB\arrow{dl}{b}\\
\Gen M
\end{tikzcd}
\]  
 
 Since we have a morphism $b$ from $M\otimes_CB$ to a module in $\Gen M$, we may use $\overline{f}$ to say there exists a morphism $c\in\text{Hom}_B(M',\Gen M)$ such that $b=c\circ f$.  
 \[
\begin{tikzcd}
 M\otimes_CB\arrow{d}[swap]{b=c\circ f}\arrow{r}{f} & M'\arrow{dl}{c}\\
\Gen M
\end{tikzcd}
\]  
 
So we have $a=b\circ h=c\circ f\circ h$.  
\[
\begin{tikzcd}
 \tau_C^{-1}\Omega_C^{-1}M\arrow{d}[swap]{a=c\circ f\circ h}\arrow{r}{ h} &M\otimes_CB\arrow{r}{f} & M'\arrow{dll}{c}\\
\Gen M 
\end{tikzcd}
\]

But $\text{Hom}_B(\tau_C^{-1}\Omega_C^{-1}M,M')=0$ by Corollary $\ref{corollary homological}$ and $a$ must be the 0 morphism.  Since $a$ was arbitrary, we conclude $\text{Hom}_B(\tau_C^{-1}\Omega_C^{-1}M,\Gen M)=0$ and our result follows.
 \end{proof}
 
 Our main result allows us to drop the assumption that $\pd_C\tau_CM\leq1$ in the special case $C$ is a tilted algebra and $B=C\ltimes E$ is the corresponding cluster-tilted algebra.

\begin{theorem} 
 \label{Partial Tilt Result}
 Let M be a partial tilting C-module.  Then M is $\tau_B$-rigid if and only if $\emph{Hom}_C(\tau_C^{-1}\Omega_C^{-1}M,\Gen M)=0$.
 \end{theorem}
 \begin{proof}
 Assume $\text{Hom}_C(\tau_C^{-1}\Omega_C^{-1}M,\Gen M)=0$.  We know $M\otimes_C E\cong\tau_C^{-1}\Omega_C^{-1}M$ by Proposition $\ref{SS Prop 4.1}$ (c).  Thus, $M$ is $\tau_B$-rigid by proposition $\ref{result}$.  Now, assume $M$ is $\tau_B$-rigid.  By Proposition $\ref{SS, Proposition 3.6}$, we have the following short exact sequence in $\mathop{\text{mod}}B$
 \[
0\rightarrow M\otimes_CE\xrightarrow{f} M\otimes_CB\xrightarrow{g} M\rightarrow 0.  
\] 
 Applying $\text{Hom}_B(-,\Gen M)$, we obtain an exact sequence
\[
\text{Hom}_B(M\otimes_CB,\Gen M)\xrightarrow{\overline{f}}\text{Hom}_B(M\otimes_CE,\Gen M)\rightarrow\text{Ext}_B^1(M,\Gen M).
\]
Since $M$ is $\tau_B$-rigid, we know $\text{Ext}_B^1(M,\Gen M)=0$ by Proposition $\ref{GenM Result}$.  Thus, $\overline{f}$ must be surjective.  Let $X\in\Gen M$.  The subjectivity of $\overline{f}$ implies, given any morphism $h\in\text{Hom}_B(M\otimes_CE,X)$, there exists a morphism $j\in\text{Hom}_B(M\otimes_CB,X)$ such that  $h=j\circ f$ in $\mathop{\text{mod}}B$.  If $h$ is non-zero,  by restriction of scalars, we have a non-zero composition $h_C=j_C\circ f_C$ in $\mathop{\text{mod}}C$.  Here $h_C$, $j_C$, and $f_C$ denote the $C$-module morphisms of $h$, $j$, and $f$ respectively.

We know $M\otimes_CE\cong\tau_C^{-1}\Omega_C^{-1}M$ by Proposition $\ref{SS Prop 4.1}$ (c).  Since $C$ is tilted, Lemma $\ref{SS 4.5}$ says $\id_C(\tau_C^{-1}\Omega_C^{-1}M)\leq1$.  Proposition $\ref{SS, Proposition 4.2}$ (a) then gives  $(\tau_C^{-1}\Omega_C^{-1}  M)\otimes_CB\cong\tau_C^{-1}\Omega_C^{-1}M$.  By Lemma $\ref{Adjunct}$, 
\[
 \text{Hom}_B(\tau_C^{-1}\Omega_C^{-1}M,M\otimes_CB)\cong\text{Hom}_C(\tau_C^{-1}\Omega_C^{-1}M,(M\otimes_CB)_C).
\] 
Here, $(M\otimes_CB)_C$ denotes the $C$-module structure of $M\otimes_CB$.  We know from Proposition $\ref{SS, Proposition 3.6}$ that, as a $C$-module, $M\otimes_CB\cong M\oplus(M\otimes_CE)$.  Again, by Proposition $\ref{SS Prop 4.1}$ (c), $M\otimes_CE\cong\tau_C^{-1}\Omega_C^{-1}M$.  So we have
\[   
\text{Hom}_C(\tau_C^{-1}\Omega_C^{-1}M,(M\otimes_CB)_C)\cong\text{Hom}_C(\tau_C^{-1}\Omega_C^{-1}M,\tau_C^{-1}\Omega_C^{-1}M\oplus M).
\]
Since $M$ is partial tilting, $\pd_CM\leq1$ and Corollary $\ref{corollary homological}$ says $\text{Hom}_C(\tau_C^{-1}\Omega_C^{-1}M,M)=0$.  Thus, 
\[
\text{Hom}_C(\tau_C^{-1}\Omega_C^{-1}M,(M\otimes_CB)_C)\cong\text{Hom}_C(\tau_C^{-1}\Omega_C^{-1}M,\tau_C^{-1}\Omega_C^{-1}M).
\]
By Lemma $\ref{Adjunct}$,
\[
\text{Hom}_C(M\otimes_CB,X)\cong\text{Hom}_C(M,(X)_C).
\]
Thus, $f_C$ is a morphism from $\tau_C^{-1}\Omega_C^{-1}M$ to itself and $j_C$ is a morphism from $M$ to $(X)_C$.  This implies the composition $j_C\circ f_C$ is 0 and contradicts $h_C$ being non-zero.  Since $h$ and $X$ were arbitrary, we conclude $\text{Hom}_B(\tau_C^{-1}\Omega_C^{-1}M,\Gen M)=0$ which implies $\text{Hom}_C(\tau_C^{-1}\Omega_C^{-1}M,\Gen M)=0$.

 \end{proof}

For an illustration of this theorem, see Examples $\ref{taufirst}$ and $\ref{tausecond}$ in section 5.
  \par
  As a corollary, we have a characterization determining when an indecomposable $\tau_C$-rigid module is also $\tau_B$-rigid.
  \begin{cor} 
  \label{Tau 1}
  Let M be an indecomposable $\tau_C$-rigid module.  Then M is $\tau_B$-rigid if and only if $\emph{Hom}_C(\tau_C^{-1}\Omega_C^{-1}M,\Gen M)=0.$
  \end{cor}
  \begin{proof}
  Since $M$ is indecomposable and $C$ is tilted, we know from Proposition $\ref{Tilting Prop}$ (e) that $M\in\mathcal{X}(T)$ or $M\in\mathcal{Y}(T)$.  Assume $M\in\mathcal{Y}(T)$.  By Proposition $\ref{Tilting Prop}$ (d), $\pd_CM\leq1$.  Since $M$ is $\tau_C$-rigid by assumption, we have $M$ is a partial tilting module.  Our result follows from Theorem $\ref{Partial Tilt Result}$.
\par
Next, assume $M\in\mathcal{X}(T)$.  Then Proposition $\ref{Tilting Prop}$ (c) says $\id_CM\leq1$.  Thus, $\tau_C^{-1}\Omega_C^{-1}M=0$ and certainly $\text{Hom}_C(\tau_C^{-1}\Omega_C^{-1}M,\Gen M)=0$.
 Also, Proposition $\ref{ID}$ says $M$ is $\tau_B$-rigid.  Our result follows. 
 \end{proof}
  The case where $M$ is a tilting $C$-module follows from the following proposition.
\begin{prop} 
\label{faithful 2}
Let M be a $\tau_C$-rigid module which is faithful.  Then M is $\tau_B$-rigid if and only if $\id_CM\leq1$.
\end{prop}
\begin{proof}
If $\id_CM\leq1$, then $M$ is $\tau_B$-rigid by Proposition $\ref{ID}$.  Conversley, assume $M$ is $\tau_B$-rigid and suppose $\id_CM=2$.  Then Lemma $\ref{Homological Result}$ (b) implies $\text{Hom}_C(M,\tau_C\Omega_C(DC))\neq0$.  Consider the following short exact sequence in $\mathop{\text{mod}}B$ guaranteed by Proposition $\ref{SS, Proposition 3.6}$ and Proposition $\ref{SS Prop 4.1}$
\[
0\rightarrow DC\rightarrow DB\xrightarrow{f}\tau_C\Omega_C(DC)\rightarrow 0.
\]
Apply $\text{Hom}_B(M,-)$ to obtain the exact sequence
\[
\text{Hom}_B(M,DB)\xrightarrow{\overline{f}}\text{Hom}_B(M,\tau_C\Omega_C(DC))\rightarrow\text{Ext}_B^1(M,DC)\rightarrow\text{Ext}_B^1(M,DB).
\]
Now, $\text{Ext}_B^1(M,DB)=0$ because $DB$ is an injective $B$-module.  Also, because $M$ is a faithful $C$-module, Lemma $\ref{faithful}$ tells us that $DC$ is generated by $M$.  Thus, because $M$ is $\tau_B$-rigid, we know $\text{Ext}_B^1(M,DC)=0$ by Proposition $\ref{GenM Result}$.  This implies that $\overline{f}$ is a surjective morphism.  Thus, given any morphism $g\in\text{Hom}_B(M,\tau_C\Omega_C(DC))$, there exists a morphism $h\in\text{Hom}_B(M,DB)$ such that $g=f\circ h$.  
\par
Next, consider an injective envelope $j\!:M\rightarrow I_0$ of $M$ in $\mathop{\text{mod}}C$.  Now, $I_0$ may or may not be an injective $B$-module but $j$ is still an injective map in $\mathop{\text{mod}}B$.  Since $DB$ is an injective $B$-module, there exists a morphism $k\!:I_0\rightarrow DB$ such that $h=k\circ j$.  
\[
\begin{tikzcd}
0\arrow{r} & M\arrow{r}{j}\arrow{d}[swap]{h} & I_0\arrow{ld}{k}\\
                 & DB
\end{tikzcd} 
\]
Thus, we have $g=f\circ h=f\circ k\circ j$.  
\[
\begin{tikzcd}
M\arrow{r}{g}\arrow{d}{j} & \tau_C\Omega_C(DC)\\
I_0\arrow{ur}[swap]{f\circ k}
\end{tikzcd}
\]
But $I_0$ is an injective $C$-module and Lemma $\ref{Homological Result}$ implies $\text{Hom}_C(I_0,\tau_C\Omega_C(DC))=0$ and subsequently $\text{Hom}_B(I_0,\tau_C\Omega_C(DC))=0$.  This forces $g=f\circ k\circ j=0$.  Since $g$ was an arbitrary morphism, we conclude $\text{Hom}_B(M,\tau_C\Omega_C(DC))=0$.  But we showed $\text{Hom}_C(M,\tau_C\Omega_C(DC))\neq0$, which implies $\text{Hom}_B(M,\tau_C\Omega_C(DC))\neq0$, and we have a contradiction.  Thus, the assumption $\id_CM=2$ must be false, and we conclude $\id_CM\leq1$.
\end{proof}
\begin{cor} Suppose M is a tilting $C$-module.  Then M is $\tau_B$-tilting if and only if $\id_CM\leq1$.
\label{Tilt}
\end{cor}
\begin{proof}
Since $M$ is a tilting $C$-module, it is faithful by Lemma $\ref{faithful}$, and our result follows from Proposition $\ref{faithful 2}$. 
\end{proof}

For an illustration of this corollary, see Examples $\ref{tauthird}$ and $\ref{tauforth}$ in section 5.
\par
We may generalize the preceding result in the special case that the algebra $C$ is tilted and $B=C\ltimes E$ is the corresponding cluster-tilted algebra.
\begin{prop} Suppose $M$ is $\tau_C$-tilting.  Then M is $\tau_B$-tilting if and only if $\id_CM\leq1$.
\end{prop}
\begin{proof}
Assume $\id_CM\leq1$.  Since $M$ is $\tau_C$-rigid, we know from Proposition $\ref{ID}$ that $M$ is also $\tau_B$-rigid.  Next, assume $M$ is $\tau_B$-tilting and suppose $\id_CM=2$.  Then at least one indecomposable summand of $M$, say $M_i$, has injective dimension equal to 2 in $\mathop{\text{mod}}C$.  By Proposition $\ref{Tilting Prop}$, we know $M_i\in\mathcal{Y}(T)$.  By Proposition $\ref{split}$, we know $(\mathcal{X}(T),\mathcal{Y}(T))$ is split which implies $\tau_CM_i\in\mathcal{Y}(T)$ and Proposition $\ref{Tilting Prop}$ gives $\pd_C\tau_CM_i\leq1$.  Thus, by Proposition $\ref{Induced}$, we have that $M_i\otimes_CB$ is $\tau_B$-rigid.
\par
By Lemma $\ref{AR submodule}$, we know $\tau_B(M_i\otimes_CB)$ is a submodule of $\tau_BM_i$.  Thus, we have $\Gen M\subseteq~^{\perp}(\tau_B(M_i\otimes_CB))$.  By Lemma $\ref{Butt Hurt 2}$, there exists an exact sequence
\[
M_i\otimes_CB\xrightarrow{f}M^0\xrightarrow{g} M^1\rightarrow 0
\]
where $f$ is a minimal left $\Gen M$-approximation of $M_i\otimes_CB$, $M^0$ and $M^1$ are in $\add M$, and we have $\add M^0\cap\add M^1=0$.  Next, consider the following short exact sequence
\[
0\rightarrow\ker g\rightarrow M^0\xrightarrow{g}M^1\rightarrow 0. 
\] 
We have a surjective morphism $f\!:M_i\otimes_CB\rightarrow\ker g$.  Using Lemma $\ref{Adjunct}$, we have a surjective morphism $f_C\!:M_i\rightarrow(\ker g)_C$ in $\mathop{\text{mod}}C$.  Since $\ker g$ is a submodule of the $C$-module $M^0$, we know $(\ker g)_C=\ker g$.  Since $\ker g\in\Gen M_i$, we have a contradiction to proposition $\ref{GenM Result}$.  Also, the sequence can not split because Lemma $\ref{Butt Hurt 2}$ guarantees $\add M^0\cap\add M^1=0$.

The only remaining possibility is $M^1=0$.  Since $M$ is sincere by Proposition $\ref{Sincere}$, we must have $M_i\otimes_CB\cong M^0$ by Lemma $\ref{extra}$.  This is clearly a contradiction and implies $\id_CM_i\leq1$.  Since $M_i$ was arbitrary, we conclude $\id_CM\leq1$.
\end{proof}

\section{Projective Covers and $\tau_B$-rigidity}
In this section, we wish to use a module's projective cover to determine whether a $C$-module is $\tau_B$-rigid.  We being with projective $C$-modules.  We derive a necessary and sufficient condition for a projective $C$-module to be $\tau_B$-rigid.
 \begin{prop} 
 \label{Projective Rigid}
 Let $P$ be a projective $C$-module with $\overline{P}$ a projective cover of $\tau_C^{-1}\Omega_C^{-1}P$ in $\mathop{\emph{mod}}C$.  Then P is $\tau_B$-rigid if and only if $\emph{Hom}_C(\overline{P},P)=0$.
 \end{prop}
 \begin{proof}
 In $\mathop{\text{mod}}B$, consider the following short exact sequence guaranteed by Proposition $\ref{SS, Proposition 3.6}$ and Proposition $\ref{SS Prop 4.1}$
 \[
 0\rightarrow\tau_C^{-1}\Omega_C^{-1}P\xrightarrow{f}P\otimes_CB\xrightarrow{g}P\rightarrow 0.
 \]
 Since $\overline{P}\otimes_CB$ is a projective cover of $\tau_C^{-1}\Omega_C^{-1}P$ in $\mathop{\text{mod}}B$ by Lemma $\ref{My Projective Cover}$, we have a minimal projective presentation
 \[
 \overline{P}\otimes_CB\xrightarrow{h}\ P\otimes_CB\xrightarrow{g}P\rightarrow 0.
 \] 
 By Proposition $\ref{Useful Observations}$, $P$ is $\tau_B$-rigid if and only if $\text{Hom}_B(P\otimes_CB,P)\xrightarrow{\overline{h}}\text{Hom}_B(\overline{P}\otimes_CB,P)$ is surjective.  Assume $\text{Hom}_C(\overline{P},P)=0$.  Considering $\overline{P}\otimes_CB$ as a $C$-module, we know $\overline{P}\otimes_CB\cong(\overline{P}\otimes_CC)\oplus(\overline{P}\otimes_CE)$.  Now, $\overline{P}\otimes_CC\cong\overline{P}$ and Proposition $\ref{SS Prop 4.1}$ implies that $\overline{P}\otimes_CE\cong\tau_C^{-1}\Omega_C^{-1}\overline{P}$.  We have $\Hom_C(\overline{P},P)=0$ by assumption and $\Hom_C(\tau_C^{-1}\Omega_C^{-1}\overline{P},P)=0$ by Lemma $\ref{Homological Result}$.  Thus, $\text{Hom}_B(\overline{P}\otimes_CB,P)=0$ and clearly $\overline{h}$ will be surjective.  We conclude $P$ is $\tau_B$-rigid.  
 \par
 Conversely, assume $P$ is $\tau_B$-rigid.  Then $\overline{h}$ must be a surjective morphism, i.e., given any morphism $j\in\text{Hom}_B(\overline{P}\otimes_CB,P)$, there exists a morphism $k\in\text{Hom}_B(P\otimes_CB,P)$ such that $j=k\circ h$.  
 \[
\begin{tikzcd}
& \overline{P}\otimes_CB\arrow{d}{j=k\circ h}\arrow{dl}[swap]{h}\\
P\otimes_CB\arrow{r}{k} & P
\end{tikzcd}
\]  
But $h$ must factor through $\tau_C^{-1}\Omega_C^{-1}P$, and $\text{Hom}_B(\tau_C^{-1}\Omega_C^{-1}P,P)=0$ by Lemma $\ref{Homological Result}$.  This implies that $j$ must be the 0 morphism, and thus $\text{Hom}_B(\overline{P}\otimes_CB,P)=0$.  Since $\overline{P}\otimes_CB$ is the projective cover of $\overline{P}$, we must have $\text{Hom}_B(\overline{P},P)=0$.  By restriction of scalars, $\Hom_C(\overline{P},P)=0$.   
\end{proof}

 \begin{prop} 
\label{Mini 1}
Let M be a $\tau_C$-rigid module with $f\!:P_0\rightarrow M$ a projective cover in $\mathop{\emph{mod}}C$.  If $\emph{Hom}_C(\tau_C^{-1}\Omega_C^{-1}P_0,\Gen M)=0$, then M is $\tau_B$-rigid.
\end{prop}
\begin{proof}
We modify the proof of Theorem $\ref{Main Ext}$ by replacing $\text{Hom}_C(\tau_C^{-1}\Omega_C^{-1}P_0,M)=0$ with the assumption $\text{Hom}_C(\tau_C^{-1}\Omega_C^{-1}P_0,\Gen M)=0$.  The concluding statement is now $\text{Ext}_B^1(M,\Gen M)=0$ and we conclude by Proposition $\ref{GenM Result}$ that $M$ is $\tau_B$-rigid. 
\end{proof}
\begin{cor} If M is $\tau_C$-rigid, and $\pd_CX\leq1$ for every module $X\in\Gen M$, then M is $\tau_B$-rigid.
\end{cor}
\begin{proof}
Since $\pd_CX\leq1$ for every module $X\in\Gen M$, $\text{Hom}_C(\tau_C^{-1}\Omega_C^{-1}C,\Gen M)=0$ by Lemma $\ref{Homological Result}$.  Our result follows from Proposition $\ref{Mini 1}$.  
\end{proof}
\begin{cor} Let M be $\tau_C$-rigid with $f\!:P_0\rightarrow M$ a projective cover in $\mathop{\emph{mod}}C$.  If $P_0$ is $\tau_B$-rigid, then M is $\tau_B$-rigid.
\end{cor}
\begin{proof}
Consider $g\!:\overline{P}\rightarrow\tau_C^{-1}\Omega_C^{-1}P_0$ a projective cover in $\mathop{\text{mod}}C$.  Since $P_0$ is $\tau_B$-rigid by assumption, we know $\text{Hom}_C(\overline{P},P_0)=0$ by Proposition $\ref{Projective Rigid}$.  Suppose there exists a morphism $h\!:\tau_C^{-1}\Omega_C^{-1}P_0\rightarrow X$ with $X\in\Gen M$.  This also gives a morphism $h\circ g\!:\overline{P}\rightarrow X$ because $\overline{P}$ is a projective $C$-module.  Since $X\in\Gen M$, we have a surjective morphism $k\!:M^d\rightarrow X$.  Combining with the fact $P_0$ is a projective cover of $M$, we have a surjective morphism $k\circ f^d\!:P_0^d\rightarrow X$.  However, since $\overline{P}$ is a projective $C$-module, we have an induced morphism $j\!:\overline{P}\rightarrow P_0^d$ such that $h\circ g=k\circ f^d\circ j$ and the following diagram commutes.  
\[
\begin{tikzcd}
& & \overline{P}\arrow{ddll}[swap]{j}\arrow{d}{g}\\
& & \tau_C^{-1}\Omega_C^{-1}P_0\arrow{d}{h}\\
P_0^d\arrow{r}{f^d} & M^d\arrow{r}{k} & X
\end{tikzcd}
\] 

But $\text{Hom}_C(\overline{P},P_0)=0$ and $j$ must be the 0 morphism.  If $g$ is non-zero then we must have that $h$ is also the 0 morphism.  Since $h$ was arbitrary, we conclude $\text{Hom}_C(\tau_C^{-1}\Omega_C^{-1}P_0,X)=0$ and Proposition $\ref{Mini 1}$ implies $M$ is $\tau_B$-rigid. 
\end{proof}

  We have the following corollary in the special case that $M$ is partial tilting and the projective dimension of $\tau_CM$ is not necessarily less than or equal to 1 nor is the algebra $C$ assumed to be tilted.
\begin{cor} Let M be a partial tilting C-module with $f\!:P_0\rightarrow M$ a projective cover in $\mathop{\emph{mod}}C$.  If $\emph{Hom}_C(\Omega_C(\tau_C^{-1}\Omega_C^{-1}P_0),M)=0$, then M is $\tau_B$-rigid.
\end{cor}
\begin{proof}
Consider the following short exact sequence in $\mathop{\text{mod}}C$
\begin{equation}\tag{1}
0\rightarrow\Omega_C^1(\tau_C^{-1}\Omega_C^{-1}P_0)\rightarrow P_1\rightarrow\tau_C^{-1}\Omega_C^{-1}P_0\rightarrow 0 
\end{equation}
where $P_1$ is a projective cover of $\tau_C^{-1}\Omega_C^{-1}P_0$.  Apply $\text{Hom}_C(-,M)$ to obtain the exact sequence
\[
\text{Hom}_C(\tau_C^{-1}\Omega_C^{-1}P_0,M)\rightarrow\text{Hom}_C(P_1,M)\rightarrow\text{Hom}_C(\Omega_C^1(\tau_C^{-1}\Omega_C^{-1}P_0),M).
\]
Since $M$ is a partial tilting module we know $\pd_CM\leq1$.  Thus, $\text{Hom}_C(\tau_C^{-1}\Omega_C^{-1}P_0,M)=0$ by Lemma $\ref{Homological Result}$.  Also, $\text{Hom}_C(\Omega_C^1(\tau_C^{-1}\Omega_C^{-1}P_0),M)=0$ by asumption.  Since the sequence is exact, we have $\text{Hom}_C(P_1,M)=0$.  Since $P_1$ is a projective $C$-module, this further implies that $\text{Hom}_C(P_1,\Gen M)=0$.  Apply $\text{Hom}_C(-,\Gen M)$ to sequence (1) to obtain the exact sequence
\[
0\rightarrow\text{Hom}_C(\tau_C^{-1}\Omega_C^{-1}P_0,\Gen M)\rightarrow\text{Hom}_C(P_1,\Gen M).
\]
Since $\text{Hom}_C(P_1,\Gen M)=0$ and the sequence is exact, $\text{Hom}_C(\tau_C^{-1}\Omega_C^{-1}P_0,\Gen M)=0$.  By Proposition $\ref{Mini 1}$, we have that $M$ is $\tau_B$-rigid.
\end{proof}
Next, we examine the special case where $M$ is a semisimple $C$-module.  We recall that a module $M$ is $\it{semi simple}$ if it is a direct sum of simple modules. 
 \begin{prop} Let M be a $\tau_C$-rigid semisimple C-module.  Consider $f\!:P_0\rightarrow M$ a projective cover and $g\!:M\rightarrow I_0$ an injective envelope in $\mathop{\emph{mod}}C$.  
 \begin{enumerate}
 \item[$\emph{(a)}$] If $\emph{Hom}_C(\tau_C^{-1}\Omega_C^{-1}P_0,M)=0$, then M is $\tau_B$-rigid.
 \item[$\emph{(b)}$] If $\emph{Hom}_C(M,\tau_C\Omega_C I_0)=0$, then M is $\tau_B$-rigid.
 \end{enumerate}
 \end{prop}
 \begin{proof}
We prove (a) with the proof of (b) being similar.  By assumption, we have $M$ is $\tau_C$-rigid and 
$\text{Hom}_C(\tau_C^{-1}\Omega_C^{-1}P_0,M)=0$.  Thus, we know from Theorem $\ref{Main Ext}$ that $M$ is a rigid B-module.  Since $M$ is semisimple, we have that $\Gen M$ = $\add M$.  Thus, we have 
\[
\text{Ext}_B^1(M,\Gen M)=\text{Ext}_B^1(M,\add M)=\text{Ext}_B^1(M,M)=0.  
\]
By Proposition $\ref{GenM Result}$, we conclude $M$ is $\tau_B$-rigid. 
 \end{proof}

\section{Examples}
In this section we illustrate our main results with several examples.  We will use the following throughout this section.  Let $A$ be the path algebra of the following quiver:
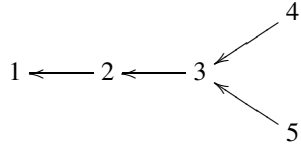
\[
\xymatrix@R10pt{&&&4\ar[dl]\\1&2\ar[l]&3\ar[l]\\ &&&5\ar[ul]}
\]
\par
Since $A$ is a hereditary algebra, we may construct a tilted algebra.  To do this, we need an $A$-module which is tilting.  Consider the Auslander-Reiten quiver of $A$ which is given by:
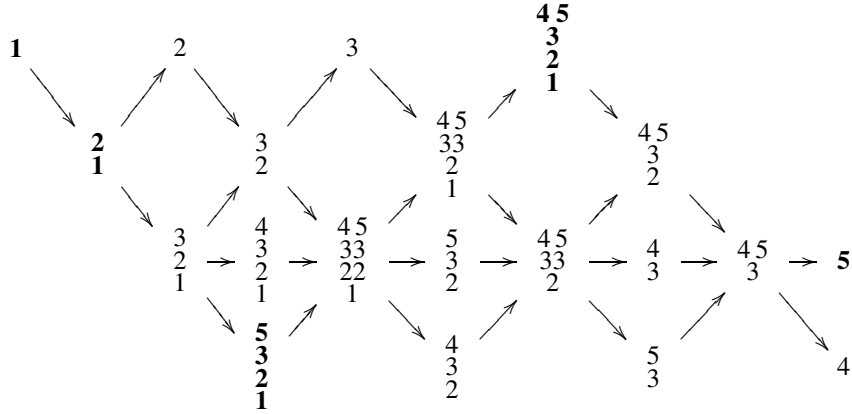
\[
\xymatrix@C=10pt@R=0pt{
{\begin{array}{c} \bf 1\end{array}}\ar[dr]&&
{\begin{array}{c} 2\end{array}}\ar[dr]&&
{\begin{array}{c} 3\end{array}}\ar[dr]&&
{\begin{array}{c}  \bf 4\ 5\\  \bf 3\\ \bf 2\\ \bf 1 \end{array}}\ar[dr]&&
\\
&{\begin{array}{c}  \bf 2\\ \bf 1\end{array}}\ar[dr]\ar[ur]&&
{\begin{array}{c} 3\\2\end{array}}\ar[dr]\ar[ur]&&
{\begin{array}{c} 4\,5\\33\\2\\1\end{array}}\ar[dr]\ar[ur]&&
{\begin{array}{c} 4\,5\\3\\2\end{array}}\ar[dr]&&
\\
&&{\begin{array}{c} 3\\2\\1\end{array}}\ar[dr]\ar[ur]\ar[r]&
{\begin{array}{c} 4\\3\\2\\1\end{array}}\ar[r]&
{\begin{array}{c} 4\,5\\33\\22\\1\end{array}}\ar[dr]\ar[ur]\ar[r]&
{\begin{array}{c} 5\\3\\2\end{array}}\ar[r]&
{\begin{array}{c} 4\,5\\33\\2\end{array}}\ar[dr]\ar[ur]\ar[r]&
{\begin{array}{c} 4\\3\end{array}}\ar[r]&
{\begin{array}{c} 4\,5\\3\end{array}}\ar[dr]\ar[r]&
{\begin{array}{c}  \bf 5\end{array}}
\\
&&&{\begin{array}{c} \bf  5\\ \bf 3\\ \bf 2\\ \bf 1\end{array}}\ar[ur]&&
{\begin{array}{c} 4\\3\\2\end{array}}\ar[ur]&&
{\begin{array}{c} 5\\3\end{array}}\ar[ur]&&
{\begin{array}{c} 4\end{array}}
}
\]

Let $T$ be the tilting $A$-module
\[
T=
{\begin{array}{c} 5\end{array}} \oplus
{\begin{array}{c} 4\,5\\3\\2\\1\end{array}} \oplus
{\begin{array}{c} 5\\3\\2\\1\end{array}} \oplus
{\begin{array}{c} 2\\1\end{array}} \oplus
{\begin{array}{c} 1\end{array}} 
\]

The corresponding titled algebra $C=\text{End}_AT$ is given by the bound quiver
$$\begin{array}{cc}
\xymatrix{1\ar[r]^\za&2\ar[r]^\zb&3\ar[r]^\zg&4\ar[r]&5}
&\quad\za\zb\zg=0\end{array}$$

Then, the Auslander-Reiten quiver of $C$ is given by: 

$$
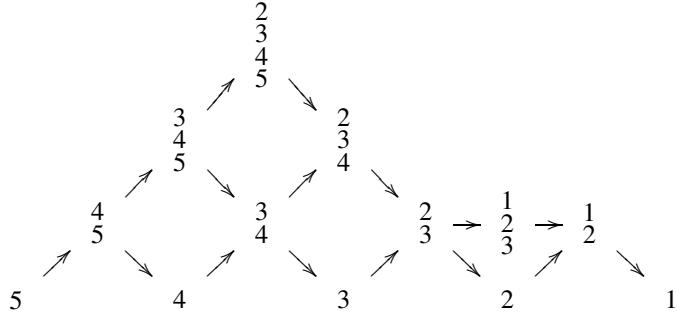
\xymatrix@C=10pt@R=0pt
{&&&{\begin{array}{c} 2\\ 3\\4\\5 \end{array}}\ar[dr]&&
&&
\\
&&{\begin{array}{c} 3\\4\\5 \end{array}}\ar[ur]\ar[dr]&&
{\begin{array}{c} 2\\ 3\\4 \end{array}}\ar[dr]&
\\
&{\begin{array}{c} 4\\5  \end{array}}\ar[ur]\ar[dr]&&
{\begin{array}{c} 3\\ 4  \end{array}}\ar[ur]\ar[dr]&&
{\begin{array}{c} 2\\3  \end{array}}\ar[r]\ar[dr]&
{\begin{array}{c} 1\\2\\ 3 \end{array}}\ar[r]&
{\begin{array}{c} 1\\2 \end{array}}\ar[dr]&&
\\
{\begin{array}{c}\ \\5\\ \  \end{array}}\ar[ur]&&
{\begin{array}{c}\ \\ 4\\ \  \end{array}}\ar[ur]&&
{\begin{array}{c}\ \\ 3\\ \  \end{array}}\ar[ur]&&
{\begin{array}{c}\ \\ 2\\ \  \end{array}}\ar[ur]&&
{\begin{array}{c}\ \\ 1\\ \  \end{array}}&&
}$$

The corresponding cluster-tilted algebra $B=C\ltimes\text{Ext}_C^2(DC,C)$ is given by the bound quiver 
$$\begin{array}{cc}
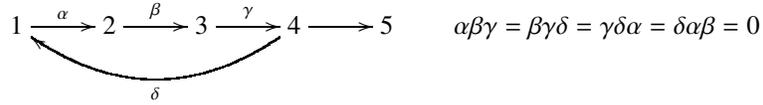
\xymatrix{1\ar[r]^\za&2\ar[r]^\zb&3\ar[r]^\zg&4\ar@/^20pt/[lll]^\zd\ar[r]&5}
&\quad \za\zb\zg=\zb\zg\zd=\zg\zd\za=\zd\za\zb=0 
\end{array}$$

Then, the Auslander-Retien quiver of $B$ is given by:

$$
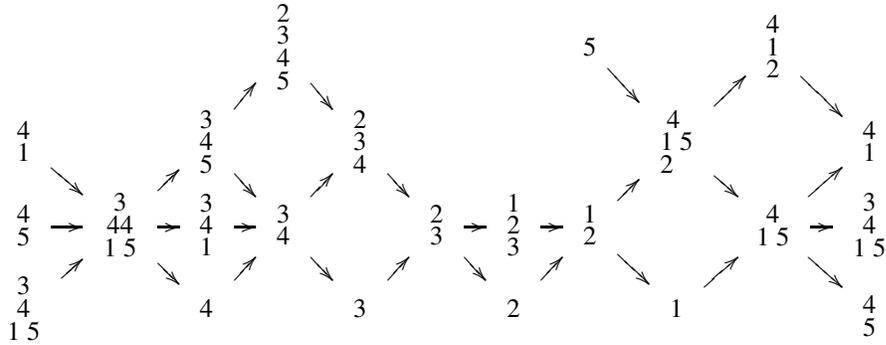
\xymatrix@C=8pt@R=0pt
{&&{\begin{array}{c} \end{array}}
&&{\begin{array}{c} 2\\ 3\\4\\5 \end{array}}\ar[dr]&&
{\begin{array}{c} \end{array}}&& 
{\begin{array}{c} 5 \end{array}}\ar[dr]&&
{\begin{array}{c} 4\\ 1\\2 \end{array}}\ar[dr]&&
{\begin{array}{c} \end{array}}
\\
&{\begin{array}{c} 4\\1 \end{array}}\ar[dr]&&
{\begin{array}{c} 3\\4\\5 \end{array}}\ar[ur]\ar[dr]&&
{\begin{array}{c} 2\\3\\4 \end{array}}\ar[dr]&&
{\begin{array}{c} \end{array}}&&
{\begin{array}{l} \ 4\\1\ 5\\2 \end{array}}\ar[ur]\ar[dr]&&
{\begin{array}{c} 4\\1 \end{array}}&&
\\
&{\begin{array}{c} 4\\ 5 \end{array}}\ar[r]&
{\begin{array}{c} 3\\44\\ 1\ 5 \end{array}}\ar[ur]\ar[dr]\ar[r]&
{\begin{array}{c} 3\\4\\1 \end{array}}\ar[r]&
{\begin{array}{c} 3\\4 \end{array}}\ar[ur]\ar[dr]&
{\begin{array}{c} \end{array}}&
{\begin{array}{c} 2\\ 3 \end{array}}\ar[dr]\ar[r]&
{\begin{array}{c} 1\\2\\ 3 \end{array}}\ar[r]&
{\begin{array}{c} 1\\2 \end{array}}\ar[dr]\ar[ur]&
{\begin{array}{c} \end{array}}&
{\begin{array}{c} 4\\1\ 5 \end{array}}\ar[dr]\ar[r]\ar[ur]&
{\begin{array}{c} 3\\4\\1\ 5 \end{array}}&&
\\
&{\begin{array}{c}3 \\ 4\\ 1\ 5 \end{array}}\ar[ur]
&&{\begin{array}{c}\ \\ 4\\ \  \end{array}}\ar[ur]&&
{\begin{array}{c}\ \\ 3\\ \  \end{array}}\ar[ur]&&
{\begin{array}{c}\ \\ 2\\ \  \end{array}}\ar[ur]&&
{\begin{array}{c}\ \\ 1\\ \  \end{array}}\ar[ru]&&
{\begin{array}{c}\ \\ 4\\ 5  \end{array}}
}$$

We will use Lemma $\ref{Homological Result}$ frequently so we note that 
\[
\tau_C^{-1}\Omega_C^{-1}C={\begin{array}{c}1\\2\end{array}}\oplus{\begin{array}{c}1\end{array}}~~,~~
\tau_C\Omega_C(DC)={\begin{array}{c}3\\4\end{array}}\oplus{\begin{array}{c}4\end{array}}.
\]
We will illustrate Theorem $\ref{Main1}$ and Proposition $\ref{Main Prop}$.  We will start with Theorem $\ref{Main1}$.
\begin{example}
\label{taufirst}
Consider the $C$-module $M={\begin{array}{c}1\\2\\3\end{array}}\oplus{\begin{array}{c}3\\4\end{array}}$.  Then $M$ is partial tilting and $\tau_C^{-1}\Omega_C^{-1}M=1$.  Since $1\in\Gen M$, we have $\text{Hom}_C(\tau_C^{-1}\Omega_C^{-1}M,\Gen M)\neq0$.  Note that $\tau_BM={\begin{array}{c}3\\44\\ 1\ 5\end{array}}$ and $\text{Hom}_B(M,\tau_BM)\neq0$ in accordance with Theorem $\ref{Main1}$. 
\end{example}
\begin{example}
\label{tausecond}
Consider the $C$-module $N={\begin{array}{c}3\\4\\5\end{array}}\oplus{\begin{array}{c}3\\4\end{array}}\oplus{\begin{array}{c}4\end{array}}$.  Then $N$ is partial tilting and $\tau_C^{-1}\Omega_C^{-1}N={\begin{array}{c}1\\2\end{array}}\oplus{\begin{array}{c}1\end{array}}$.  It is easily seen that $\text{Hom}_C(\tau_C^{-1}\Omega_C^{-1}N,\Gen N)=0$.  We note that $\tau_BN={\begin{array}{c}3\\44\\ 1\ 5\end{array}}\oplus{\begin{array}{c}3\\4\\1\ 5\end{array}}\oplus{\begin{array}{c}4\\1\end{array}}$ and $\text{Hom}_B(N,\tau_BN)=0$ in accordance with Theorem $\ref{Main1}$.
\end{example}
The next two examples will illustrate Proposition $\ref{Main Prop}$.
\begin{example}
\label{tauthird}
Consider the tilting $C$-module 
\[
M={\begin{array}{c}4\end{array}}\oplus{\begin{array}{c}4\\5\end{array}}\oplus{\begin{array}{c}3\\4\\5\end{array}}\oplus{\begin{array}{c}1\\2\\3\end{array}}\oplus{\begin{array}{c}2\\3\\4\\5\end{array}}.
\]
Recall that $\tau_C\Omega_C(DC)={\begin{array}{c}3\\4\end{array}}\oplus{\begin{array}{c}4\end{array}}$.  Since $\text{Hom}_C(M,\tau_C\Omega_C(DC))\neq0$, Lemma $\ref{Homological Result}$ says $\id_CM=2$.  Note that $\tau_BM={\begin{array}{c}1\end{array}}\oplus{\begin{array}{c}4\\1\end{array}}\oplus{\begin{array}{c}3\\4\\1\ 5\end{array}}$ and we have $\text{Hom}_B(M,\tau_BM)\neq0$ in accordance with Corollary $\ref{Tilt}$.
\end{example}
\begin{example}
\label{tauforth}
Consider the tilting $C$-module
\[
T={\begin{array}{c}2\end{array}}\oplus{\begin{array}{c}2\\3\end{array}}\oplus{\begin{array}{c}2\\3\\4\end{array}}\oplus{\begin{array}{c}1\\2\\3\end{array}}\oplus{\begin{array}{c}2\\3\\4\\5\end{array}}.
\]
Since $\text{Hom}_C(T,\tau_C\Omega_C(DC))=0$, Lemma $\ref{Homological Result}$ says $\id_CT\leq1$.
We note that 
\[
\tau_CT\cong\tau_BT={\begin{array}{c}3\end{array}}\oplus{\begin{array}{c}3\\4\end{array}}\oplus{\begin{array}{c}3\\4\\5\end{array}} 
\]
and $\text{Hom}_B(T,\tau_BT)=0$ in accordance with Corollary $\ref{Tilt}$.
\end{example}
\begin{example} 
In Proposition $\ref{Main Prop}$, the condition $M$ is $\tau_C$-tilting is necessary.  If we only assume $M$ is support $\tau_C$-tilting, the statement is no longer true.  Consider the support $\tau_C$-tilting module
\[
M={\begin{array}{c}5\end{array}}\oplus{\begin{array}{c}4\\5\end{array}}\oplus{\begin{array}{c}3\\4\\5\end{array}}.
\]
Here, $\id_CM=2$ yet $M$ is support $\tau_B$-tilting.
\end{example}

\noindent Department of Mathematics, University of Connecticut-Waterbury, Waterbury, CT 06702, USA
\\
\it{E-mail address}: \bf{stephen.zito@uconn.edu}

\end{document}